\documentclass[11pt]{elsarticle}
\usepackage{mathtools,amsmath,amssymb,amsthm}
\usepackage{mathrsfs}
\usepackage{graphicx}
\usepackage{tikz}
\usetikzlibrary{cd} 
\usetikzlibrary{positioning}
\usepackage{enumerate}
\usepackage{float}
\usepackage[colorlinks=true,linkcolor=blue,citecolor=blue,urlcolor=blue]{hyperref}
\usetikzlibrary{calc}

\theoremstyle{plain}
\newtheorem{theorem}{Theorem}[section]
\newtheorem{lemma}[theorem]{Lemma}
\newtheorem{proposition}[theorem]{Proposition}
\theoremstyle{definition}
\newtheorem{definition}[theorem]{Definition}
\newtheorem{remark}[theorem]{Remark}
\newtheorem{corollary}[theorem]{Corollary}
\newtheorem{example}[theorem]{Example}

\makeatletter
\newcommand{\defeq}{\coloneqq}

\newcommand{\cl}{\mathrm{cl}}

\newcommand{\Rmnum}[1]{\expandafter@slowromancap\romannumeral #1@}
\makeatother

\newcommand{\da}{{\downarrow}}
\newcommand{\ua}{{{\uparrow}}}
\makeatletter
\def\ps@pprintTitle{%
  \let\@oddhead\@empty
  \let\@evenhead\@empty
  \def\@oddfoot{\reset@font\hfil\thepage\hfil}
  \let\@evenfoot\@oddfoot
}\makeatother

\begin{document}
\begin{frontmatter}

\title{Coherence of Smyth powerspaces}
\tnotetext[t1]{Supported by NSF of China (No.12371457).}

\author{Rongqi Xiao}
\ead{rongqixiao2024@163.com}

\author{Xiaodong Jia\texorpdfstring{\corref{a1}}{}}
\ead{jiaxiaodong@hnu.edu.cn}
\cortext[a1]{Corresponding author.}
\address{School of Mathematics, Hunan University, Changsha, Hunan, 410082, China}

\begin{abstract}
In this paper, we study when the Smyth powerspace $\mathcal{Q}^*_v(X)$ of a topological space $X$ is coherent, and prove that $X$ is coherent and weakly Hausdorff if and only if $\mathcal{Q}^*_v(X)$ is coherent and weakly Hausdorff. We give examples to show that neither coherence nor weak Hausdorffness of $X$ solely implies that $\mathcal{Q}^*_v(X)$ is coherent or weakly Hausdorff. As a byproduct, our work gives an affirmative answer to a question raised by Xu~\cite[Question 7.6]{xu2026}.


\end{abstract}

\end{frontmatter}

\section{Introduction}
The Smyth powerspace $\mathcal{Q}^*_v(X)$ of a topological space $X$ consists of all nonempty compact saturated subsets of~$X$ and equips with the upper Vietoris topology. It is one of the classical powerspace constructions in non-Hausdorff topology, and especially in domain theory that is used for denotational semantics of demonic nondeterminism in theoretical computer science~\cite{abramsky94, gierz03, goubault13a}. 

To serve as semantic models, one will need the Smyth powerspace be defined on the corresponding semantic category. This purpose has led to many research in this direction. That is, given a space $X$, one asks which properties of $X$ are preserved by the construction $\mathcal{Q}^*_v(X)$. 
Compactness of spaces is trivially in that preservation list. 
Goubault-Larrecq showed that if $X$ is stably compact, then $\mathcal{Q}^*_v(X)$ is stably compact~\cite{goubault2010-degrootduality}; Heckmann and Keimel showed that $X$ is sober if and only if $\mathcal{Q}^*_v(X)$ is sober \cite[Theorem 3.13]{HECKMANN2013215}; Xu, Xi and Zhao showed that $X$ is well-filtered if and only if that $\mathcal{Q}^*_v(X)$ is well-filtered~\cite{xuxizhao2021}; and Lyu, Chen and the second author showed that $X$ is locally compact if and only if that $\mathcal{Q}^*_v(X)$ is core-compact if and only if $\mathcal{Q}^*_v(X)$ is locally compact~\cite{LyuChenJia2022}. As stably compact spaces are compact, locally compact, sober and \textit{coherent} spaces, and motivated by these developments, we naturally ask whether coherence of a space can be preserved by the Smyth powerspace construction. The main goal of this paper is to answer this question.



A space $X$ is called to be \textit{coherent} if the intersection of arbitrary two compact saturated subsets of $X$ is again compact saturated. Coherence plays an important role in non-Hausdorff topology and domain theory, where it is used to characterize Lawson compactness \cite{gierz03, jia16a} and to locate Cartesian closed subcategories of domains \cite{jung89}.
Our first main result shows that the Smyth powerspace $\mathcal{Q}^*_v(X)$ of a coherent topological space $X$ is indeed coherent if  $X$ is additionally assumed to be \textit{weakly Hausdorff}. 
Moreover, we prove that $X$ is coherent and weakly Hausdorff if and only if its Smyth powerspace $\mathcal{Q}^*_v(X)$ is coherent and weakly Hausdorff (Theorem~\ref{thm:coherence-of-Qv}). The proof benefits from the weak Hausdorffness assumption, a concept introduced by Keimel and Lawson~\cite{keimel05} and further studied by Goubault-Larrecq~\cite{goubault22weakly}, providing a natural framework in which intersections of compact saturated sets behave well. 


This result has an immediate consequence in the Hausdorff case. Since Hausdorff spaces are weakly Hausdorff, coherent and well-filtered, we obtain that $\mathcal{Q}^*_v(X)$ is coherent whenever $X$ is Hausdorff. Combining this with the known preservation result for well-filteredness~\cite{xuxizhao2021} and Xu's criterion that coherent well-filtered spaces are strongly well-filtered~\cite{xu2026}, we obtain that $\mathcal{Q}^*_v(X)$ is strongly well-filtered for every Hausdorff space $X$, giving an affirmative answer to a question raised by Xu~\cite[Question 7.6]{xu2026}. 

Additionally, we give examples to show that neither coherence nor weak Hausdorffness of $X$ solely implies that $\mathcal{Q}^*_v(X)$ is coherent or weakly Hausdorff. Example~\ref{ex:X-weakly-Hausdorff-Q(X)not} gives a weakly Hausdorff space with a non weakly Hausdorff Smyth powerspace. 
In Section~4, we construct an explicit coherent $T_1$ space $X$ that is not weakly Hausdorff, and its Smyth powerspace $\mathcal{Q}^*_v(X)$ fails to be coherent. These two examples show that our main theorem (Theorem~\ref{thm:coherence-of-Qv}) cannot be sharpened further. 

Our results suggests that weak Hausdorffness and coherence is closely connected, and somewhat bonded together by the Smyth powerspace construction. 


\section{Preliminaries}
 
 The following definitions can be found in \cite{abramsky94, gierz03, goubault13a}.

Let $X$ be a topological space. We denote by $\mathcal{O}(X)$ the set of all open subsets of $X$, ordered by the inclusion order $\subseteq$. The \emph{specialization preorder} $\le_{\mathrm{spec}}$ on $X$ is defined by $x\le_{\mathrm{spec}} y$ if and only if $x\in \cl(\{y\})$, where $\cl(\{y\})$ is the closure of $\{y\}$. The relation $\le_{\mathrm{spec}}$ is a partial order if and only if $X$ is $T_0$. 
Intersections of open subsets of $X$ are called \emph{saturated} subsets. We say $X$ is \emph{coherent} if the intersection of any two compact saturated subsets of $X$ is compact saturated. A $T_0$ space $X$ is called \textit{well-filtered} if for any filtered family of compact saturated subsets $K_i, i\in I$ and open subset $U$, that $\bigcap_{i\in I}K_i \subseteq{U}$ holds implies that $K_i\subseteq{U}$ holds for some $i\in I$. 

Write $\mathcal{Q}(X)$ for the set of all compact saturated subsets of $X$.
The upper Vietoris topology $\nu$ on $\mathcal{Q}(X)$ is generated by the basic opens
\[
\square U \;:=\; \{K\in \mathcal{Q}(X): K\subseteq U\},
\]
where $U$ ranges over the open subsets of $X$. We write $\mathcal{Q}_v(X):=(\mathcal{Q}(X),\nu)$ as the \emph{Smyth powerspace} of $X$, and use $\mathcal{Q}^*_v(X)$ to denote the subspace of $\mathcal{Q}_v(X)$ that consists of \textit{nonempty} compact saturated subsets of $X$ and equipped with the subspace topology.  
Note that the specialization order on $\mathcal{Q}^*_v(X)$ (or on $\mathcal{Q}_v(X)$) is \emph{reverse inclusion}:
for $K,L\in \mathcal{Q}(X),  K \le_{\mathrm{spec}} L\iff L\subseteq K$.

\section{The weakly Hausdorff case}

Weak Hausdorffness is a topological property initially introduced by Keimel and Lawson in~\cite{keimel05}. 

\begin{definition}
    $X$ is \emph{weakly Hausdorff} if for all compact saturated subsets $Q_1,Q_2$ of $X$, and for every open neighborhood $W$ of $Q_1\cap Q_2$, there are two open neighborhoods $U$ of $Q_1$ and $V$ of $Q_2$ such that $U\cap V\subseteq W$.
\end{definition}

\begin{example}
    \begin{enumerate}
        \item All Hausdorff spaces are weakly Hausdorff, and in fact, a topological space $X$ is Hausdorff if and only if it is $T_1$ and weakly Hausdorff (see \cite{goubault22weakly}). 
        \item Every poset in its Alexandroff topology (the topology consisting of all upper sets) is weakly Hausdorff (see \cite{goubault22weakly}). 
        \item Coherent locally compact sober spaces are weakly Hausdorff (see \cite{keimel05}). 
    \end{enumerate}
\end{example}

\begin{remark}\label{remark-adding-top}
    For a topological space $X$, we add a distinguished point $\top$ to $X$, and equipped $X\cup\{\top\}$ with the topology consisting of $\varnothing$ and all $U\cup \{\top \}$, where $U$ are open sets in $X$. We denote the resulting space as $X^\top$. Since  $\{\top \}$ is open and $\top$ is the largest element in the specialization order of $X^\top$, each compact saturated subsets of $X^\top$ is either empty or of the form $K\cup \{\top\}$, where $K$ is compact saturated in $X$. 
    Then one could easily verify that:
    \begin{itemize}
        \item $X$ is coherent if and only if $X^\top$ is coherent;
        \item $X$ is weakly Hausdorff if and only if $X^\top$ is weakly Hausdorff. 
    \end{itemize}
Hence we have that $\mathcal{Q}^*_v(X)$ is coherent (reps., weakly Hausdorff) if and only if $\mathcal{Q}_v(X)$ is coherent (reps., weakly Hausdorff) for some topological space $X$. This is because $\mathcal{Q}_v(X)$ is homeomorphic to $(\mathcal{Q}^*_v(X))^\top$, with the empty compact saturated set in $\mathcal{Q}_v(X)$ being in correspondence to $\top$. 
\end{remark}

The switch between $\mathcal{Q}^*_v(X)$ and $\mathcal{Q}_v(X)$ enables us to talk about intersections of compact saturated sets without worrying about the intersection being empty.

\begin{lemma}[Continuity of intersection]\label{lem:cap-continuous}
Assume that $X$ is coherent and weakly Hausdorff. Then the map
\[
\cap \;:\; \mathcal{Q}_v(X)\times \mathcal{Q}_v(X) \longrightarrow \mathcal{Q}_v(X),\qquad (K,L)\longmapsto K\cap L
\]
is well-defined and continuous.
\end{lemma}
\begin{proof}
\emph{Well-definedness.} If $K,L\in \mathcal{Q}(X)$, then $K\cap L$ is saturated, as the intersection of saturated sets is always saturated. 
By coherence $K\cap L$ is compact. Hence $K\cap L\in \mathcal{Q}(X)$.

\emph{Continuity.} It suffices to show that the preimage of each basic open set $\square U\subseteq \mathcal{Q}_v(X)$ is open
in $\mathcal{Q}_v(X)\times \mathcal{Q}_v(X)$. For an open $U\subseteq X$, we have 
\[
\cap^{-1}(\square U)
=\{(K,L)\in \mathcal{Q}(X)\times \mathcal{Q}(X): K\cap L\subseteq U\}.
\]
Let $(K_0,L_0)$ lie in this set, i.e., $K_0\cap L_0\subseteq U$. By weak Hausdorffness, there exist opens $V,W\subseteq X$
such that $K_0\subseteq V$, $L_0\subseteq W$, and $V\cap W\subseteq U$. Then
$(K_0,L_0)\in \square V\times \square W$, and for any $(K,L)\in \square V\times \square W$ we have
$K\subseteq V$ and $L\subseteq W$, hence $K\cap L\subseteq V\cap W\subseteq U$, i.e., $(K,L)\in \cap^{-1}(\square U)$.
Thus$
(K_0,L_0)\in \square V\times \square W \subseteq \cap^{-1}(\square U)$;
and therefore,
$\cap^{-1}(\square U)$ is open and $\cap$ is continuous. 
\end{proof}

\begin{proposition}\label{prop:Qcoh/wh-Xcoh/wh}
Let $X$ be a topological space. 
\begin{enumerate}
    \item If $\mathcal{Q}_v(X)$ (or $\mathcal{Q}^*_v(X)$) is coherent, then $X$ is coherent. 
    \item If $\mathcal{Q}_v(X)$ (or $\mathcal{Q}^*_v(X)$) is weakly Hausdorff, then $X$ is weakly Hausdorff.  
\end{enumerate}
\end{proposition}
\begin{proof}
1. This result is already shown in ~\cite[Theorem 3.39]{luli2021}; note that only the coherence of $\mathcal{Q}_v(X)$ is used there.

2. Assume that $\mathcal{Q}_v(X)$ is weakly Hausdorff. To prove $X$ is weakly Hausdorff, we pick $Q_1,Q_2\in \mathcal{Q}_v(X)$ and let $W\in \mathcal{O}(X)$ be such that  $Q_1\cap Q_2\subseteq W$. We define $\mathcal K_i:=\{K\in \mathcal{Q}_v(X):K\subseteq Q_i\}$ for $i=1,2$. Each $\mathcal K_i$, as a principal filter, is of course compact saturated in $\mathcal{Q}_v(X)$, and we have $\mathcal K_1\cap\mathcal K_2\subseteq \square W$. Since $\mathcal{Q}_v(X)$ is weakly Hausdorff, there exist open subsets $\mathcal U,\mathcal V\subseteq \mathcal{Q}_v(X)$ such that $\mathcal K_1\subseteq\mathcal U$, $\mathcal K_2\subseteq\mathcal V$, and $\mathcal U\cap\mathcal V\subseteq \square W$. Now consider the continuous map 
\[
\xi \;:\; X\longrightarrow \mathcal{Q}_v(X),\qquad x\longmapsto {\uparrow} x
\] 
and let $U:=\xi^{-1}(\mathcal U)$, $V:=\xi^{-1}(\mathcal V)$.  The sets $U$ and $V$ are open. Since $x\in Q_i$ implies ${\uparrow} x\subseteq Q_i$, i.e., $\xi(x)\in\mathcal K_i$, for $i=1,2$, we have $Q_1\subseteq U$ and $Q_2\subseteq V$. We conclude by showing that $U\cap V\subseteq{W}$. Indeed, if $x\in U\cap V$, then ${\uparrow} x=\xi(x)\in\mathcal U\cap\mathcal V\subseteq \square W$, so ${\uparrow} x\subseteq W$, and hence $x\in W$. 
\end{proof}

We arrive at the main result of this section. 

\begin{theorem}\label{thm:coherence-of-Qv}
Let $X$ be a topological space. The following are equivalent. 
\begin{enumerate}
    \item $X$ is coherent and weakly Hausdorff;
    \item $\mathcal{Q}_v(X)$ is coherent and weakly Hausdorff;
    \item $\mathcal{Q}^*_v(X)$ is coherent and weakly Hausdorff. 
\end{enumerate}
\end{theorem}

\begin{proof}
Assume that $X$ is coherent and weakly Hausdorff. We show that $1$ implies $2$.
We first show that $\mathcal{Q}_v(X)$ is coherent. 
Let $\mathcal A,\mathcal B\subseteq \mathcal{Q}_v(X)$ be compact saturated subsets.
We must show that $\mathcal A\cap\mathcal B$ is compact in $\mathcal{Q}_v(X)$. Since $\mathcal A$ and $\mathcal B$ are compact, $\mathcal A\times \mathcal B$ is compact in $\mathcal{Q}_v(X)\times \mathcal{Q}_v(X)$.
By Lemma~\ref{lem:cap-continuous}, the image
$\cap(\mathcal A\times\mathcal B) = \{K\cap L : K\in \mathcal A, L\in \mathcal B \}$ is compact in $\mathcal{Q}_v(X)$. We claim that $\cap(\mathcal A\times\mathcal B)=\mathcal A\cap\mathcal B$.
Take $K\in \mathcal A$ and $L\in \mathcal B$. Since $K\cap L\subseteq K$ and
$K\cap L\subseteq L$, the specialization order on $\mathcal{Q}_v(X)$ (reverse inclusion) yields
$K \le_{\mathrm{spec}} (K\cap L)$ and 
$L \le_{\mathrm{spec}} (K\cap L)$. 
Because $\mathcal A$ and $\mathcal B$ are saturated, they are upward closed with respect to $\le_{\mathrm{spec}}$.
Hence $(K\cap L)\in \mathcal A$ and $(K\cap L)\in \mathcal B$, so $K\cap L\in \mathcal A\cap\mathcal B$.
Therefore $\cap(\mathcal A\times\mathcal B)\subseteq \mathcal A\cap\mathcal B$. Conversely, if $M\in \mathcal A\cap\mathcal B$, then $(M,M)\in\mathcal A\times\mathcal B$ and
$\cap(M,M)=M$. So $M\in \cap(\mathcal A\times\mathcal B)$, and hence $\mathcal A\cap\mathcal B\subseteq
\cap(\mathcal A\times\mathcal B)$.
Thus $\mathcal A\cap\mathcal B=\cap(\mathcal A\times\mathcal B)$, which is compact. Since the intersection of saturated
sets is saturated, $\mathcal A\cap\mathcal B$ is compact saturated. This shows that $\mathcal{Q}_v(X)$ is coherent.\\
To show that $\mathcal{Q}_v(X)$ is weakly Hausdorff, we let $\mathcal A,\mathcal B\subseteq \mathcal{Q}_v(X)$ be compact saturated and let $\mathcal W\in \mathcal{O}(\mathcal{Q}_v(X))$ satisfy $\mathcal A\cap\mathcal B\subseteq\mathcal W$. Then $\mathcal A\cap \mathcal B = \cap(\mathcal A\times  \mathcal B) \subseteq \mathcal W$, and hence   $\mathcal A\times  \mathcal B \subseteq\cap^{-1}(\mathcal W)$. Note also that  $\cap^{-1}(\mathcal W)$ is open in $\mathcal{Q}_v(X)\times \mathcal{Q}_v(X)$ as the $\cap$ operation is continuous by Lemma~\ref{lem:cap-continuous}. By Wallace's generalized tube lemma (see for example Exercise~9 on Page~171 in~\cite{munkres99}), there are open sets $\mathcal U,\mathcal V\subseteq \mathcal{Q}_v(X)$ such that $\mathcal A\subseteq\mathcal U$, $\mathcal B\subseteq\mathcal V$, and $\mathcal U\times\mathcal V\subseteq\cap^{-1}(\mathcal W)$. Moreover, if $M\in\mathcal U\cap\mathcal V$, then $(M,M)\in\mathcal U\times\mathcal V$, and $M=\cap(M,M)\in\mathcal W$; therefore, we have $\mathcal U\cap\mathcal V\subseteq\mathcal W$, and $\mathcal{Q}_v(X)$ is weakly Hausdorff.

That $2$ implies $1$ is shown in Proposition~\ref{prop:Qcoh/wh-Xcoh/wh}, and the equivalence between $2$ and $3$ is explained in Remark~\ref{remark-adding-top}. 
\end{proof}

According to Xu~\cite{xu2026}, a ($T_0$) space $X$ is called \textit{strongly well-filtered} if for any filtered family of compact saturated subsets $K_i, i\in I$, any compact saturated subset $K$, and open subset $U$, that $\bigcap_{i\in I}(K_i\cap K) \subseteq{U}$ holds implies that $K_i\cap K \subseteq{U}$ already holds for some $i\in I$. 
Xu asked in~\cite[Question 7.6]{xu2026} whether the space $\mathcal{Q}^*_v(X)$ is strongly well-filtered when $X$ is Hausdorff. 
The following result, which is a corollary to Theorem~\ref{thm:coherence-of-Qv}, provides an affirmative answer.

\begin{corollary}
If $X$ is Hausdorff, then the Smyth powerspace $\mathcal{Q}^*_v(X)$ is strongly well-filtered.
\end{corollary}
\begin{proof}
Every Hausdorff space is weakly Hausdorff, well-filtered and coherent. Then $\mathcal{Q}^*_v(X)$ is well-filtered by~\cite[Theorem 4]{xuxizhao2021}.
Now Theorem~\ref{thm:coherence-of-Qv} tells us that  $\mathcal{Q}^*_v(X)$ is coherent, and by~\cite[Proposition 4.3]{xu2026} coherent well-filtered spaces are strongly well-filtered. 
\end{proof}

In the next example, we will see that without $X$ being coherent assumed,  $\mathcal{Q}_v(X)$, or equivalently, $\mathcal{Q}^*_v(X)$, could fail to be weakly Hausdorff even if $X$ is weakly Hausdorff. 

\begin{example}\label{ex:X-weakly-Hausdorff-Q(X)not}
    Let $C = \{c_0, c_1, c_2 \cdots\}$ be a countable set, with elements indexed by natural numbers, and $E= C \cup\{a,b\}$ be ordered by $a<c_n$ and $b<c_n$ for every $n\in\mathbb N$, with $a,b$ incomparable and the $c_n's$ pairwise incomparable, and equip $E$ with the Alexandroff topology. Let $I$ be the interval $[0,2]$ with the usual topology and put $X=E\sqcup I$, the disjoint topological sum of $E$ and $I$. Then $X$ is weakly Hausdorff but $\mathcal{Q}_v(X)$ is not weakly Hausdorff.
\medskip
\begin{figure}[H]
\begin{center}

\begin{tikzpicture}[x=1cm,y=1cm,line cap=round,line join=round]
  \def\figlw{0.8pt}
  \tikzset{
    edge/.style={line width=\figlw},
    pt/.style={circle,draw=black,fill=white,minimum size=5.5pt,inner sep=0pt,line width=\figlw},
    dot/.style={circle,fill=black,inner sep=0.7pt}
  }

  \coordinate (a) at (1.7,0);
  \coordinate (b) at (4.8,0);

  \node[pt] (aNode) at (a) {};
  \node[pt] (bNode) at (b) {};
  \node[below=5pt] at (aNode) {$a$};
  \node[below=5pt] at (bNode) {$b$};

  \foreach \x/\name in {0.0/n1,1.6/n2,3.2/n3,5.9/n4}{
    \node[pt] (\name) at (\x,2.05) {};
    \draw[edge] (\name) -- (aNode);
    \draw[edge] (\name) -- (bNode);
  }

  \node[dot] at (4.05,2.05) {};
  \node[dot] at (4.38,2.05) {};
  \node[dot] at (4.71,2.05) {};
  \node[dot] at (6.45,2.05) {};
  \node[dot] at (6.78,2.05) {};


  \coordinate (L) at (9,0.90);
  \coordinate (R) at (12.4,0.90);

  \draw[edge]
    (L) -- ($(L)+(0,-0.35)$)
        -- ($(R)+(0,-0.35)$)
        -- (R);

  \node[below=7pt] at ($(L)+(0,-0.35)$) {$0$};
  \node[below=7pt] at ($(R)+(0,-0.35)$) {$2$};
\end{tikzpicture}
\\Figure 1: $X$ in Example \ref{ex:X-weakly-Hausdorff-Q(X)not}
\end{center}
\end{figure}
\end{example}
\begin{proof}
    The space $X$ is weakly Hausdorff. Indeed, the topology on $E$ is the Alexandroff topology, so every compact saturated subset of $E$ is open, hence $E$ is trivially  weakly Hausdorff. The interval~$I$ is Hausdorff, hence weakly Hausdorff. Weak Hausdorffness is preserved by finite topological sums: Assume that $Q_i=A_i\sqcup K_i$ $(i=1,2)$ with $A_i\subseteq E$, $K_i\subseteq I$ are compact, and $W=W_E\sqcup W_I$ is open with $Q_1\cap Q_2\subseteq W$. As $E$ and $I$ are weakly Hausdorff, choose opens $U_E,V_E$ in $E$ with $A_1\subseteq U_E, A_2\subseteq{V_E}$ and $U_E\cap V_E\subseteq{W_E}$, and opens $U_I,V_I$ in $I$ with $K_1\subseteq{U_I}, K_2\subseteq{V_I}$ and $U_I \cap V_I\subseteq{W_I}$, then 
    $Q_1\subseteq U_E\sqcup U_I$, $Q_2\subseteq V_E\sqcup V_I$ and 
    $(U_E\sqcup U_I)\cap(V_E\sqcup V_I)\subseteq W$.

    We proceed to show that $Q_v(X)$ is not be weakly Hausdorff.
    Set $A={\uparrow}_E a=\{a\}\cup C$, $B={\uparrow}_E b=\{b\}\cup C$, $K=[0,1]$, $L=[1,2]$, $P=A\sqcup K$, and $Q=B\sqcup L$. Here $A$ and $B$ are principal upper sets in the Alexandroff space $E$, hence compact saturated; $K$ and $L$ are compact in $I$; therefore $P,Q\in \mathcal Q(X)$. For each $m\in\mathbb N$ let $F_m=\{c_0,\dots,c_m\}\subseteq C$ and $J_m=(1-\frac1{m+1},1+\frac1{m+1})\subseteq [0,2]$, and define $\mathcal W=\bigcup_{m\in\mathbb N}\Box(F_m\sqcup J_m)$, which is obviously open in $\mathcal Q_v(X)$.

    Note that ${\uparrow} P\cap{\uparrow} Q=\{M\in \mathcal Q_v(X):M\subseteq P\cap Q= C\sqcup\{1\}\}$. If $M\subseteq C \sqcup\{1\}$ is compact saturated, write $M=M_E\sqcup M_I$ with $M_E\subseteq C$ and $M_I\subseteq\{1\}$. As the singletons $\{c_n\}, n\in \mathbb N$ are open, the subspace topology on $C$ is discrete and then the compact $M_E$ must be finite. Hence $M\subseteq F_m\sqcup J_m$ for some~$m$, so $M\in\mathcal W$. Thus, $\mathcal W$ is an open neighborhood of ${\uparrow} P\cap{\uparrow} Q$.

    Now take arbitrary basic neighborhoods $\Box O_1$ of $P$ and $\Box O_2$ of $Q$, and we claim that $\Box O_1 \cap \Box O_2$ cannot be contained in $\mathcal M$. For this, we write $O_1=U\sqcup U'$ and $O_2=V\sqcup V'$. We have $A\subseteq U\subseteq E$, $K\subseteq U'\subseteq I$, $B\subseteq V\subseteq E$, and $L\subseteq V'\subseteq I$. Then $C = A\cap B \subseteq U\cap V$, and $U'\cap V'$ is an open neighborhood of $1$ in $I$; choose $t\in(U'\cap V')\setminus\{1\}$, and then choose an $m$ with $t\notin J_m$. Put $M=F_m\sqcup\{1,t\}$. Then $M\in \mathcal Q(X)$ and $M\subseteq O_1\cap O_2$, i.e., $M\in\Box O_1\cap\Box O_2$. However $M\notin\mathcal W$. For otherwise $M\subseteq F_k\sqcup J_k$ for some $k$, and we would have that  $F_m\subseteq F_k$. So we know that $k\ge m$; and therefore, $J_k\subseteq J_m$. But this would imply that $t\in J_k\subseteq J_m$, in contradiction to the choice of $m$.

    As every pair of basic neighborhoods of $P$ and $Q$ meet outside $\mathcal W$, every pair of neighborhoods of $P$ and $Q$ meet outside $\mathcal W$ too, and hence $\mathcal Q_v(X)$ cannot be weakly Hausdorff.
\end{proof}

\section{A counterexample without the weakly Hausdorff assumption}
 
In this section, we give a coherent topological space $X$, but its Smyth powerspace  $\mathcal{Q}^*_v(X)$ fails to be coherent. Hence 
this example shows that the weak Hausdorffness assumption in Theorem~\ref{thm:coherence-of-Qv} cannot be removed.

Let
\[
\mathbb N=\{0,1,2,\dots\},\qquad
X^+=\{n^+:n\in\mathbb N\},\qquad
X^-=\{n^-:n\in\mathbb N\},
\]
and put
\[
X=X^+\sqcup X^-.
\]
We define a topology $\tau$ on $X$ by declaring that a subset $U\subseteq X$ is open if either
\begin{itemize}
    \item $U=\varnothing$, or
    \item $X\setminus U$ is finite, or
    \item $U\subseteq X^+$ and $X^+\setminus U$ is finite.
\end{itemize}

One could easily verify that $\tau$ is a topology, and $X$ satisfies the following properties. Throughout this section, $X$ is always fixed as this topological space. 

\begin{remark}
    The construction of topological space $X$ is analogous to the space in \cite[Example~4.1]{goubault24a}, where that space was constructed to show  that $\omega$-projective limits of coherent spaces need not be coherent. The two spaces share a same underlying set, but are actually \textit{not} homeomorphic. Indeed, the next proposition shows that our space $X$ is coherent, while the space  in~\cite[Example~4.1]{goubault24a} is designed to fail coherence. 
\end{remark}

\begin{proposition}\label{prop:X-coherent}
The space $X$ is $T_1$, second-countable, every subset of $X$ is compact, and hence $X$ is compact, locally compact and coherent.
\end{proposition}

\begin{proof}
For each $x\in X$, the complement $X\setminus\{x\}$ is cofinite in $X$, hence open; and therefore, $X$ is $T_1$.  Moreover, the family $\mathcal{B}:=\{X\setminus F : F\subseteq X \text{ is finite}\}\cup
\{X^+\setminus F : F\subseteq X \text{ is finite}\}$
is a countable base for $X$. So $X$ is second-countable.

Now assume $A\subseteq X$ is nonempty and we prove that $A$ is compact. For this we let $\mathcal U$ be an open cover of $A$. 

If $A\cap X^-\neq\varnothing$, choose $a\in A\cap X^-$. There exists $U\in\mathcal U$ with $a\in U$. Since $a$ is a negative point, $U$ must be of the second type of opens; hence $U$ is cofinite in $X$. Therefore $A\setminus U$ is finite. For each point of $A\setminus U$, choose one member of $\mathcal U$ to cover it, and we obtain a finite subcover of~$A$. 

If $A\subseteq X^+$, choose $a\in A$, and take $U\in\mathcal U$ with $a\in U$. Then $U$ is either cofinite in $X$ or cofinite in $X^+$. In either case, $A\setminus U$ is finite, and again $\mathcal A$ has a finite subcover of $A$. 

Thus, every subset of $X$ is compact. Hence $X$ is compact, locally compact and coherent.
\end{proof}

\begin{proposition}\label{prop:not-weh}
The space $X$ is not weakly Hausdorff.
\end{proposition}

\begin{proof}
Let
\[
Q_1:=X^+\cup\{(2n+1)^-:n\in\mathbb N\},
\qquad
Q_2:=X^+\cup\{(2n)^-:n\in\mathbb N\}.
\]
By Proposition~\ref{prop:X-coherent}, $Q_1$ and $Q_2$ are compact saturated, and their intersection $Q_1\cap Q_2=X^+$ is open in $X$.

Let $U\supseteq Q_1$ and $V\supseteq Q_2$ be any two open sets. Since $Q_1$ and $Q_2$ both contain negative points, neither $U$ nor $V$ can be of the third type opens. Hence both $U$ and $V$ are cofinite in $X$, and so is $U\cap V$. In particular, $U\cap V$ contains negative points, which means $U\cap V\nsubseteq X^+$, and this fails weak Hausdorffness. 
\end{proof}

Since every nonempty subset of $X$ is compact saturated, we have $\mathcal{Q}^{\ast}(X)=2^X\setminus\{\varnothing\}$. For $x\in X$, we define
$\mathcal U_x:=\{K\in \mathcal{Q}^{\ast}(X):x\notin K\}$.

\begin{lemma}\label{lem:subbasis}
The family
$\{\mathcal U_x:x\in X\}\cup\{\Box X^+\}$
is a subbase of the upper Vietoris topology on~$\mathcal{Q}^{\ast}(X)$.
\end{lemma}

\begin{proof}
As every nonempty open subset of $X$ is of the form $X\setminus F$ or 
$X^+\setminus F$, where $F\subseteq X$ is finite, we know that 
\[
\square(X\setminus F)=\bigcap_{x\in F}\mathcal U_x,
\qquad
\square(X^+\setminus F)=\Box X^+ \cap\bigcap_{x\in F}\mathcal U_x.
\]
Moreover $\square\varnothing=\varnothing$, and the claim follows.
\end{proof}

For $K\in \mathcal{Q}^{\ast}(X) = 2^X\setminus\{\varnothing\}$, we define
\[
K^+:= K\cap X^+,
\qquad
K^-:= K\cap X^{-},
\]
and for $x\in X$, we let
\[
\mathcal F_x:=\mathcal{Q}^{\ast}(X)\setminus \mathcal U_x=\{K\in \mathcal{Q}^{\ast}(X):x\in K\},
\quad
\mathcal N:=\mathcal{Q}^{\ast}(X)\setminus \Box X^+=\{K\in \mathcal{Q}^{\ast}(X):K^-\neq\varnothing\}.
\]

\begin{lemma}[A specialized Alexander criterion]\label{lem:alexander}
For a subset $\mathcal K\subseteq \mathcal{Q}^{\ast}(X)$, the following are equivalent:
\begin{itemize}
    \item[(i)] $\mathcal K$ is a compact subset of $\mathcal{Q}^{\ast}_v(X)$;
    \item[(ii)] the following two conditions hold:
    \begin{itemize}
        \item[(K0)] whenever $S\subseteq X$ has the property that for every finite $F\subseteq S$ there exists $K_F\in\mathcal K$ with $F\subseteq K_F$, then there exists $K\in\mathcal K$ with $S\subseteq K$;
        \item[(K1)] whenever $S\subseteq X$ has the property that for every finite $F\subseteq S$ there exists $K_F\in\mathcal K$ such that $F\subseteq K_F$ and $K_F^-\neq\varnothing$, then there exists $K\in\mathcal K$ with $S\subseteq K$ and $K^-\neq\varnothing$.
    \end{itemize}
\end{itemize}
\end{lemma}

\begin{proof}
By Lemma~\ref{lem:subbasis}, the family $\{\mathcal F_x:x\in X\}\cup\{\mathcal N\}$ is a subbase of closed sets of $\mathcal{Q}^{\ast}_v(X)$. By Alexander's subbase theorem, $\mathcal K$ is compact if and only if every family of the form
\[
\{\mathcal K\cap \mathcal F_x:x\in S\}
\qquad\text{or}\qquad
\{\mathcal K\cap \mathcal N\}\cup\{\mathcal K\cap \mathcal F_x:x\in S\}
\]
with the finite intersection property has a nonempty intersection.

If the family with the finite intersection property  is of the first type, then the compactness of~$\mathcal K$ is
is equivalent to \textup{(K0)}. This is because in this case the finite intersection property reads as for each finite $F$,
\[
\bigcap_{x\in F}(\mathcal K\cap \mathcal F_x)=\left \{ K\in \mathcal{K}: K\in \bigcap_{x\in F} \mathcal F_x\right \}\neq\varnothing,\]
which is equivalent to the existence of a $K_F\in\mathcal K$ with $F\subseteq K_F$; the fact that the total intersection is nonempty means 
\[
\bigcap_{x\in S}(\mathcal K\cap \mathcal F_x)=\left \{ K\in \mathcal{K}: K\in \bigcap_{x\in S} \mathcal F_x\right \}\neq\varnothing,\]
which is equivalent to saying that there exists
$K\in\mathcal K$ such that $S\subseteq K$.

Similarly, if the family with the finite intersection property is of the second type, the corresponding nonempty intersection property is equivalent to \textup{(K1)}. 
\end{proof}

Now we construct a non-compact saturated subset $\mathcal C$ of $\mathcal{Q}^{\ast}_v(X)$  which can be written as an intersection of two compact saturated sets $\mathcal A$ and $\mathcal B$, and this means that $\mathcal{Q}^{\ast}_v(X)$ is not coherent. 

 Let
\[
O^-:=\{(2n+1)^-:n\in\mathbb N\},
\qquad
E^-:=\{(2n)^-:n\in\mathbb N\},
\]
and  we define
\[
\mathcal C:=\Bigl\{K\in \mathcal{Q}^{\ast}(X):K^+=\varnothing\ \text{or}\ K^-=\varnothing\ \text{or}\ (\forall i^+\in K^+)(\forall j^-\in K^-)\ i<j\Bigr\},
\]
\[
\mathcal A:=\mathcal C\cup\{K\in \mathcal{Q}^{\ast}(X):K\cap O^-=\varnothing\},
\qquad
\mathcal B:=\mathcal C\cup\{K\in \mathcal{Q}^{\ast}(X):K\cap E^-=\varnothing\}.
\]

\begin{proposition}\label{prop:A-B-compactsaturated-C-saturated}
The subsets $\mathcal A$ and $\mathcal B$ are compact saturated in $\mathcal{Q}^{\ast}_v(X)$ and $\mathcal C$ is saturated in~$\mathcal{Q}^{\ast}_v(X)$.
\end{proposition}

\begin{proof}
The specialization order on $\mathcal{Q}^{\ast}_v(X)$ is reverse inclusion. Hence a subset of $\mathcal{Q}^{\ast}(X)$ is saturated if and only if it is downward closed with respect to set inclusion. The definitions above show immediately that $\mathcal A$, $\mathcal B$, and $\mathcal C$ are saturated.

We prove that $\mathcal A$ is compact, and the verification of compactness of $\mathcal B$ is analogous.

By Lemma~\ref{lem:alexander}, it suffices to verify \textup{(K0)} and \textup{(K1)} for $\mathcal K=\mathcal A$.

\noindent\emph{Verification of \textup{(K0)}.}
Let $S\subseteq X$, and assume that every finite $F\subseteq S$ is contained in some member of $\mathcal A$.
\begin{itemize}
    \item If $S\cap O^- = \varnothing$, then either $S=\varnothing$, in which case we may take
    $S\subseteq \{0^+\}\in \mathcal C  \subseteq \mathcal A$, or else $S\neq\varnothing$, in which case $S \subseteq{S}\in \mathcal A$ by definition.
    \item Assume now that $S$ contains some odd negative point, say $o^-\in S\cap O^-$. We claim that $S\in\mathcal C$. Suppose not. Then there exist $p^+\in S$ and $e^-\in S$ such that $e\le p$. Consider the finite set
    \[
    F:=\{o^-,e^-,p^+\}\subseteq S.
    \]
    By assumption we find some $K\in\mathcal A$ with $F\subseteq K$. Since $o^-\in K$, $K\cap O^- \not= \varnothing$. This set $K$ must be in $\mathcal C$. But $e^-\in K$ and $p^+\in K$ with $e\le p$, which contradicts the definition of $\mathcal C$. This contradiction shows that $S\subseteq{S}\in\mathcal C\subseteq\mathcal A$.
\end{itemize}

\noindent\emph{Verification of \textup{(K1)}.}
Let $S\subseteq X$, and assume that every finite $F\subseteq S$ is contained in some $K_F\in\mathcal A$ with $K_F^- \neq\varnothing$.
\begin{itemize}
    \item If $S\cap O^- = \varnothing$, then either $S^-\neq\varnothing$, in which case $S\in\mathcal A\cap\mathcal N$, or else $S\subseteq X^+$, in which case  $S \subseteq{S\cup\{0^-\}}\in\mathcal A\cap\mathcal N.$ Thus \textup{(K1)} holds in this case.
    \item Assume now that $S$ contains some odd negative point, say $o^-\in S\cap O^-$. We show that $S\in\mathcal C\cap\mathcal N$. That $S\in \mathcal N$ is obvious. To prove $S\in\mathcal C$, we assume that $S^+ \neq \varnothing$ and let $p^+\in S$ and $e^-\in S$ be arbitrary. Consider the finite set
    \[
    F:=\{o^-,e^-,p^+\}\subseteq S.
    \]
    By assumption on $S$, there exists $K\in\mathcal A\cap\mathcal N$ with $F\subseteq K$. Since $o^-\in K$, the set $K$ cannot belong to $\{L\in \mathcal{Q}^{\ast}(X):L\cap O^-=\varnothing\}$. Hence $K\in\mathcal C$, and therefore $p<e$. Since $p^+\in S$ and $e^-\in S$ were arbitrary, we conclude that every positive index occurring in $S$ is strictly smaller than every negative index occurring in $S$. Thus $S\in\mathcal C\cap\mathcal N$.
\end{itemize}

So by Lemma~\ref{lem:alexander} $\mathcal A$ is compact. The proof for the compactness of $\mathcal B$ is similar, with odd and even negatives interchanged.
\end{proof}

\begin{proposition}\label{prop:C-not-compact}
We have $\mathcal A\cap\mathcal B=\mathcal C$ and the subset $\mathcal C$ is not compact in $\mathcal{Q}^{\ast}_v(X)$.
\end{proposition}

\begin{proof}
That $\mathcal A\cap\mathcal B=\mathcal C$ is immediate. We show that $\mathcal{C}$ is not compact. To this end, recall that $\mathcal U_x:=\{K\in \mathcal{Q}^{\ast}(X):x\notin K\}$, and we
%
consider the family of open subsets 
\[
\{\Box X^+, U_{n^+}:n\in\mathbb N\}.
\]
We claim that it covers $\mathcal C$. Let $K\in\mathcal C$.

If $K^-=\varnothing$, then $K\subseteq X^+$, so $K\in \Box X^+$.

If $K^-\neq\varnothing$, let $m:= \min \{n: n^- \in K^-\}$. By the defining property of $\mathcal C$, every $i^+\in K^+$ satisfies $i<m$. Hence $m^+ \notin K^+$, i.e., $m^+\notin K$. Therefore $K\in U_{m^+}$.

Thus $\mathcal C\subseteq \Box X^+ \cup\bigcup_{n\in\mathbb N}U_{n^+}$.

Now let
\[
\{ \Box X^+,\ U_{0^+},\dots,U_{N^+}\}
\]
be any finite subfamily. Set
\[
K_N:=\{0^+,1^+,\dots,N^+,(N+1)^-\}.
\]
Then $K_N\in\mathcal C$, because every positive index in $K_N$ is strictly smaller than the unique negative index $N+1$. However, $K_N\notin \Box X^+$, and $K_N\notin U_{i^+}$ for each $0\le i\le N$ because $i^+\in K_N$. Hence this finite subfamily, and likewise any other finite subfamily, fails to cover $\mathcal C$. So $\mathcal C$ is not compact.
\end{proof}

\begin{theorem}\label{thm:counterexample}
There exists a coherent topological space $X$ such that $\mathcal{Q}^{\ast}_v(X)$ is not coherent.
In particular, the weakly Hausdorff assumption in Theorem~\ref{thm:coherence-of-Qv} cannot be dropped.
\end{theorem}

\begin{proof}
By Proposition~\ref{prop:X-coherent}, the space $X$ introduced in this section is coherent. By Proposition~\ref{prop:A-B-compactsaturated-C-saturated}, the subsets $\mathcal A$ and $\mathcal B$ are compact saturated in $\mathcal{Q}^{\ast}_v(X)$. Their intersection $\mathcal C$ is not compact by Proposition~\ref{prop:C-not-compact}. Hence, $\mathcal{Q}^{\ast}_v(X)$ is not coherent.
\end{proof}

\section{A Scott space counterexample}
In this section, 
we give an example to show that the weak Hausdorffness assumption on $X$ in Theorem~\ref{thm:coherence-of-Qv} cannot be removed, even $X$ is assumed as a Scott space, namely a dcpo with its Scott topology. We recall that a dcpo is a poset in which every directed subset has a supremum, and the Scott topology on a dcpo consists of upper subsets that are also inaccessible by suprema of directed subsets. For a dcpo $P$, we use $\Sigma P$ to denote the associated Scott space obtained by equipping $P$ with its Scott topology. 
For detailed definitions of dcpo's and Scott topology, the reader is referred to \cite{gierz03}.

\begin{example}\label{ex:scott-counterexample}
Let
\[
M^+=\{p_n:n\in\mathbb N\},\qquad M^-=\{q_n:n\in\mathbb N\},\qquad M=M^+\cup M^-.
\]
Add two families
\[
A=\{a_{i,k}:i,k\in\mathbb N\},\qquad B=\{b_{i,k}:i,k\in\mathbb N\}.
\]
Let $P=M\cup A\cup B$, ordered by the reflexive-transitive closure of
\[
a_{i,k}\leq a_{i,\ell}\leq p_i,\qquad b_{i,k}\leq b_{i,\ell}\leq q_i\quad(k\leq \ell),
\]
\[
a_{i,k}\leq p_m\quad(m\geq k),\qquad b_{i,k}\leq p_m,q_m\quad(m\geq k).
\]
The order on $P$ can be depicted as in Figure~2. Then $P$ is a dcpo. The Scott space $\Sigma P$ is coherent, and $\mathcal{Q}_v^*(\Sigma P)$ is not coherent.

\begin{figure}[H]
\centering
\begin{tikzpicture}[x=1cm,y=1cm,line cap=round,line join=round, scale=0.8]

  \def\figlw{0.8pt}
  \def\ytop{6.2}
  \def\yone{3.9}
  \def\ytwo{2.4}
  \def\ythree{0.9}
  \def\r{0.28}

  \def\xpA{0.0}
  \def\xpB{1.8}
  \def\xpC{3.6}

  \def\xqA{8.6}
  \def\xqB{10.4}
  \def\xqC{12.2}

  \def\Lleft{-0.55}
  \def\Lright{5.2}

  \def\Rleft{7.95}
  \def\Rright{13.95}

  \tikzset{
    edge/.style={line width=\figlw},
    guide/.style={line width=\figlw,dashed},
    op/.style={circle,draw=black,fill=white,minimum size=5.5pt,inner sep=0pt,line width=\figlw},
    every node/.style={font=\small}
  }

  \newcommand{\opoint}[2]{\node[op] at (#1,#2) {};}

  \newcommand{\band}[4]{%
    \draw[edge] ($(#1,#2)+(0,#4)$) arc[start angle=90,end angle=270,radius=#4];
    \draw[edge] ($(#1,#2)+(0,#4)$) -- (#3,{#2+#4});
    \draw[edge] ($(#1,#2)+(0,-#4)$) -- (#3,{#2-#4});
    \draw[guide] (#3,{#2+#4}) -- ++(1.1,0);
    \draw[guide] (#3,{#2-#4}) -- ++(1.1,0);
  }

  \draw[guide] (4.9,\ytop) -- (7.0,\ytop);
  \draw[guide] (12.95,\ytop) -- (14.7,\ytop);

  \band{\Lleft}{\yone}{\Lright}{\r}
  \band{\Lleft}{\ytwo}{\Lright}{\r}
  \band{\Lleft}{\ythree}{\Lright}{\r}

  \band{\Rleft}{\yone}{\Rright}{\r}
  \band{\Rleft}{\ytwo}{\Rright}{\r}
  \band{\Rleft}{\ythree}{\Rright}{\r}

  \draw[guide] (\xpA,\yone) -- (\xpA,\ytop);
  \draw[guide] (\xpB,\yone) -- (\xpB,\ytop);
  \draw[guide] (\xpC,\yone) -- (\xpC,\ytop);

  \draw[edge] (\xpA,\ythree) -- (\xpA,\yone);
  \draw[edge] (\xpB,\ythree) -- (\xpB,\yone);
  \draw[edge] (\xpC,\ythree) -- (\xpC,\yone);

  \draw[guide] (\xqA,\yone) -- (\xqA,\ytop);
  \draw[guide] (\xqB,\yone) -- (\xqB,\ytop);
  \draw[guide] (\xqC,\yone) -- (\xqC,\ytop);

  \draw[edge] (\xqA,\ythree) -- (\xqA,\yone);
  \draw[edge] (\xqB,\ythree) -- (\xqB,\yone);
  \draw[edge] (\xqC,\ythree) -- (\xqC,\yone);

  \draw[edge] (\xpA,\ytop) -- ({\xpA+0.72},{\ythree+\r});
  \draw[edge] (\xpB,\ytop) -- ({\xpB+0.58},{\ytwo+\r});
  \draw[edge] (\xpC,\ytop) -- ({\xpC+0.42},{\yone+\r});

  \draw[edge] (\xqA,\ytop) -- ({\xqA+0.72},{\ythree+\r});
  \draw[edge] (\xqB,\ytop) -- ({\xqB+0.58},{\ytwo+\r});
  \draw[edge] (\xqC,\ytop) -- ({\xqC+0.42},{\yone+\r});

  \draw[edge] (\xpA,\ytop) -- ({\Rleft+0.95},{\ythree+\r});
  \draw[edge] (\xpB,\ytop) -- (\Rleft,{\ytwo+\r});
  \draw[edge] (\xpC,\ytop) -- (\Rleft,{\yone+\r});

  \node[above] at (\xpA,\ytop) {$p_1$};
  \node[above] at (\xpB,\ytop) {$p_2$};
  \node[above] at (\xpC,\ytop) {$p_3$};

  \node[above] at (\xqA,\ytop) {$q_1$};
  \node[above] at (\xqB,\ytop) {$q_2$};
  \node[above] at (\xqC,\ytop) {$q_3$};

  \opoint{\xpA}{\ytop}
  \opoint{\xpB}{\ytop}
  \opoint{\xpC}{\ytop}

  \opoint{\xqA}{\ytop}
  \opoint{\xqB}{\ytop}
  \opoint{\xqC}{\ytop}

  \foreach \x in {\xpA,\xpB,\xpC}{
    \opoint{\x}{\yone}
    \opoint{\x}{\ytwo}
    \opoint{\x}{\ythree}
  }

  \foreach \x in {\xqA,\xqB,\xqC}{
    \opoint{\x}{\yone}
    \opoint{\x}{\ytwo}
    \opoint{\x}{\ythree}
  }

\end{tikzpicture}
\\Figure 2: $P$ in Example~\ref{ex:scott-counterexample}
\end{figure}
\end{example}
\begin{proof}
    Either a directed set in $P$ has a largest element, or it is cofinal in one of the chains $(a_{i,k})_k$ or $(b_{i,k})_k$, whose suprema are $p_i$ or $q_i$, respectively. Hence $P$ is a dcpo. 
    
    To see that $\Sigma P$ is coherent, we first realize that the set $M$ of maximal points of $P$, in the induced Scott topology, is a homeomorphic copy of the space $X$ constructed in Section~4, along the canonical homeomorphism $\eta:X\to M$ by $\eta(n^+)=p_n$ and $\eta(n^-)=q_n$.
    
    We use $[M]^{<\omega}$ to denote the set of finite subsets of $M$.
    Indeed, if a Scott open set $U$ contains $p_i$, then $p_i=\bigvee_k a_{i,k}$, so $a_{i,k}\in U$ for some $k$, and hence $p_m\in U$ for all $m\geq k$. Similarly, if $q_i\in U$, then $q_i=\bigvee_k b_{i,k}$, so $p_m,q_m\in U$ for all sufficiently large $m$. Conversely, for finite $F\subseteq M$ and finite $G\subseteq M^+$, the sets $P\setminus\da F$ and $(A\cup M^+)\setminus\da G$ are Scott open, and they intersect $M$ at $M\setminus F$ and $M^+\setminus G$, respectively. Thus the subspace topology on $M$ is $\{\emptyset\}\cup\{M\setminus F:F\in[M]^{<\omega}\}\cup\{M^+\setminus G:G\in[M^+]^{<\omega}\}$, which shows that $M$ is indeed a copy of~$X$. 

    This allows us to characterize all compact saturated subsets of $\Sigma P$, with a similar analysis as in \cite[Example 2.6.1]{jia2018meet}. They are of the form $\ua F \cup S$, where $F$ is a finite subset of $A \cup B$ and $S$ is an arbitrary subset of $M$. The intersection of any two subsets of this form is again of the same form, and hence $\Sigma P$ is coherent. 

    It remains to prove that $\mathcal{Q}_v^*(\Sigma P)$ is not coherent. Put $\mathcal{R}\defeq\{K\in\mathcal{Q}_v^*(\Sigma P):K\subseteq M\}$. The subspace $\mathcal{R}$ is canonically homeomorphic to $\mathcal{Q}_v^*(M)$ and a subset of $\mathcal{R}$ is compact saturated if and only if it is compact saturated in the ambient space $\mathcal{Q}_v^*(\Sigma P)$. Let $\mathcal{A},\mathcal{B},\mathcal{C}\subseteq\mathcal{Q}_v^*(X)$ be the three families used in Theorem~\ref{thm:counterexample}, where $\mathcal{A}$ and $\mathcal{B}$ are compact saturated, $\mathcal{A}\cap\mathcal{B}=\mathcal{C}$ is not compact. The homeomorphism $\eta$ induces a homeomorphism $\widehat{\eta}:\mathcal{Q}_v^*(X)\to\mathcal{R}$, given by $\widehat{\eta}(K)=\eta[K]$. Therefore, $\widehat{\eta}[\mathcal{A}]$ and $\widehat{\eta}[\mathcal{B}]$ are compact saturated subsets of $\mathcal{Q}_v^*(\Sigma P)$, while $\widehat{\eta}[\mathcal{A}]\cap\widehat{\eta}[\mathcal{B}]=\widehat{\eta}[\mathcal{C}]$ is not compact. Hence $\mathcal{Q}_v^*(\Sigma P)$ is not coherent.
\end{proof}

\section{Closing remarks}

It is worth emphasizing that the counterexample constructed in Section~4 has very strong classical properties, as witnessed by Proposition~\ref{prop:X-coherent}. Thus, even within the class of compact, coherent,  $T_1$, locally compact and second-countable
spaces, the Smyth powerspace $\mathcal{Q}^*_v(X)$ need not preserve coherence. In other words, among these properties, the essential missing condition to make  $\mathcal{Q}^*_v(X)$ coherent is precisely weak Hausdorffness.
This shows that weak Hausdorffness is not merely a technical assumption in Theorem~\ref{thm:coherence-of-Qv}, but is relevant to the preservation of coherence by the Smyth powerspace construction. 
A natural direction for further work is to characterize those coherent spaces~$X$ for which $\mathcal{Q}^*_v(X)$ is coherent, and to determine whether weak Hausdorffness can be replaced by other assumptions.  

We conclude this paper with listing the following, where we use "$\implies$" to mean implication. 
\begin{enumerate}
    \item $\mathcal{Q}^*_v(X)$ is coherent $\implies$ $X$ is coherent (Proposition~\ref{prop:Qcoh/wh-Xcoh/wh}, item 1);
    \item $\mathcal{Q}^*_v(X)$ is weakly Hausdorff $\implies$ $X$ is weakly Hausdorff (Proposition~\ref{prop:Qcoh/wh-Xcoh/wh}, item 2);
    \item $X$ is weakly Hausdorff $\not \Rightarrow$ $\mathcal{Q}^*_v(X)$ is weakly Hausdorff (Example~\ref{ex:X-weakly-Hausdorff-Q(X)not});
    \item $X$ is coherent $\not \Rightarrow$ $\mathcal{Q}^*_v(X)$ is coherent (Theorem~\ref{thm:counterexample});
    \item $X$ is coherent and weakly Hausdorff $\iff$ $\mathcal{Q}^*_v(X)$ is coherent and weakly Hausdorff (Theorem~\ref{thm:coherence-of-Qv}).
\end{enumerate}

\section{Acknowledgment}

We would like to thank Jean Goubault-Larrecq for useful discussion, who already knew the usefulness of weak Hausdorffness in analyzing coherence of Smyth powerspaces of topological spaces, and for pointing the Example~4.1 in \cite{goubault24a} to us. 

\bibliographystyle{plain}

\end{document}